\documentclass[11pt,reqno,tbtags,draft]{amsart}
\usepackage{amssymb}
\usepackage{url}
\usepackage[square,numbers]{natbib}
\bibpunct[, ]{[}{]}{;}{n}{,}{,}

\title
{On the tails of the limiting Quicksort distribution}

\date{28 August, 2015; minor revision 26 September, 2015}

\author{Svante Janson}
\thanks{Partly supported by the Knut and Alice Wallenberg Foundation}
\address{Department of Mathematics, Uppsala University, PO Box 480,
SE-751~06 Uppsala, Sweden}
\email{svante.janson@math.uu.se}
\newcommand\urladdrx[1]{{\urladdr{\def~{{\tiny$\sim$}}#1}}}
\urladdrx{http://www2.math.uu.se/~svante/}

\subjclass[2010]{} 

\numberwithin{equation}{section}

\renewcommand\le{\leqslant}
\renewcommand\ge{\geqslant}

\allowdisplaybreaks

\theoremstyle{plain}
\newtheorem{theorem}{Theorem}[section]
\newtheorem{lemma}[theorem]{Lemma}

\theoremstyle{definition}

\newtheorem{remark}[theorem]{Remark}

\newtheorem*{ack}{Acknowledgement}

\theoremstyle{remark}

\newenvironment{romenumerate}[1][0pt]{
\addtolength{\leftmargini}{#1}\begin{enumerate}
 }{\end{enumerate}}

\newcounter{oldenumi}
{\setcounter{oldenumi}{\value{enumi}}
\begin{romenumerate} \setcounter{enumi}{\value{oldenumi}}}
{\end{romenumerate}}

\newcounter{thmenumerate}
\newenvironment{thmenumerate}
{\setcounter{thmenumerate}{0}%
 \def\item{\par
 \refstepcounter{thmenumerate}\textup{(\roman{thmenumerate})\enspace}}
}
{}

\newcounter{xenumerate}   

\newcommand{\refL}[1]{Lemma~\ref{#1}}

\newcommand{\refS}[1]{Section~\ref{#1}}

\newcommand\marginal[1]{\marginpar[\raggedleft\tiny #1]{\raggedright\tiny#1}}
\newcommand\REM[1]{{\raggedright\texttt{[#1]}\par\marginal{XXX}}}

\begingroup
  \divide\count255 by 60
  \multiply\count255 by -60
  \advance\count255 by \time
  \ifnum \count255 < 10 \xdef\klockan{\the\count1.0\the\count255}
  \else\xdef\klockan{\the\count1.\the\count255}\fi
\endgroup

\newcommand\set[1]{\ensuremath{\{#1\}}}

\newcommand\bigpar[1]{\bigl(#1\bigr)}
\newcommand\Bigpar[1]{\Bigl(#1\Bigr)}

\newcommand\xcpar[1]{\{#1\}}

\def\rompar(#1){\textup(#1\textup)}    

\def\xexp(#1){e^{#1}}
\newcommand\ceil[1]{\lceil#1\rceil}

\newcommand\ntoo{\ensuremath{{n\to\infty}}}

\newcommand\xtoo{\ensuremath{{x\to\infty}}}

\newcommand\punkt{.\spacefactor=1000}    

\newcommand\eg{e.g\punkt}

\newcommand\eqd{\overset{\mathrm{d}}{=}}

\newcounter{CC}
\newcounter{cc}
\newcommand{\cc}{\stepcounter{cc}\ccx} 
\newcommand{\ccx}{c_{\arabic{cc}}}     
\newcommand{\ccdef}[1]{\xdef#1{\ccx}}     

\newcommand\E{\operatorname{\mathbb E{}}}
\renewcommand\P{\operatorname{\mathbb P{}}}

\newcommand\Var{\operatorname{Var}}

\newcommand\gd{\delta}

\newcommand\gam{\gamma}

\newcommand\kk{\kappa}

\newcommand\eps{\varepsilon}

\renewcommand\phi{\xxx}  

\newcommand\qw{^{-1}}

\newcommand\qqw{^{-1/2}}

\newcommand\intoi{\int_0^1}

\newcommand\intoooo{\int_{-\infty}^\infty}
\newcommand\oi{[0,1]}

\newcommand\dd{\,\mathrm{d}}

\newcommand\U{\mathsf{U}}
\newcommand\lniiw{\frac{1}{\ln2}}

\begin{document}

\begin{abstract} 
We give asymptotics for the left and right tails of the limiting Quicksort
distribution. The results agree with, but are less precise than, earlier
non-rigorous results by Knessl and Spankowski.
\end{abstract}

\maketitle

\section{Introduction}\label{S:intro}

Let $X_n$ be the number of comparisons used by the algorithm Quicksort 
when
sorting $n$ distinct numbers, initially in a uniformly random order.
Equivalently, $X_n$ is the internal pathlength in a random binary search
tree with $n$ nodes. (See \eg{} \citet[Sections 5.2.2 and 6.2.2]{KnuthIII} 
or 
\citet[Chapter 8 and Section 1.4.1]{Drmota} 
for a description of the algorithm and 
of binary search trees.)
It follows that $X_n$
satisfies the distributional recurrence relation
\begin{equation}\label{Xn}
X_n \eqd X_{U_n - 1} + X^*_{n - U_n} + n - 1,\qquad n \geq 1,  
\end{equation}
where~$\eqd$ denotes equality in distribution, and, on the right,
$U_n$ is distributed uniformly on the set $\{1, \ldots, n\}$,
$X_j^* \eqd X_j$, $X_0=0$,
and $U_n, X_0, \dots,\allowbreak X_{n - 1},\allowbreak X^*_0,\allowbreak
\dots, X^*_{n - 1}$ 
are all independent.
(Thus, \eqref{Xn} can be regarded as a definition of $X_n$.)

It is well-known, and easy to show from \eqref{Xn}, 
that 
\begin{equation}
  \E X_n=2(n+1)H_n-4n\sim 2n\ln n,  
\end{equation}
where $H_n:=\sum_{k=1}^n k\qw$ is the $n$:th harmonic number.
Moreover, it was proved by 
R\'{e}gnier~\cite{Reg} 
and 
R\"{o}sler~\cite{Roesler}, using different methods,
that the normalized variables
\begin{equation}
Z_n := \frac{X_n - \E X_n}n
\end{equation}
converge in distribution to some limiting random variable $Z$, as \ntoo.

There is no simple description of the distribution of $Z$, but various
results have been shown by several different authors. For example, 
$Z$ has an everywhere finite moment generating function, and thus all
moments are finite 
\cite{Roesler}, 
with $\E Z=0$ and $\Var Z=7-\frac{2}3\pi^2$;
furthermore, $Z$ has a
density which is infinitely differentiable \cite{TanH,SJ131}.
Moreover, the recurrence relation \eqref{Xn} yields in the limit
a distributional identity, 
which can be written as
\begin{equation}\label{recZ}
  Z\eqd UZ'+(1-U)Z''+g(U),
\end{equation}
where $U$, $Z'$ and $Z''$ are independent, $U\sim \U(0,1)$ is uniform, 
$Z',Z''\eqd Z$, and $g$ is the deterministic function
\begin{equation}
  \label{g}
g(u):=2u\ln u+2(1-u)\ln(1-u)+1.
\end{equation}
Furthermore, \citet{Roesler} showed that \eqref{recZ} together with $\E Z=0$
and $\Var Z<\infty$
determines the distribution of $Z$ uniquely; see further \cite{SJ134}.
The identity \eqref{recZ} is the basis of much of the study of $Z$,
including the present work.

In the present paper we study the asymptotics of the tail probabilities
$\P(Z\le -x)$ and $\P(Z\ge x)$ as \xtoo.
Using non-rigorous methods from applied mathematics (assuming an as yet 
unverified regularity hypothesis),
\citet{KnSz} found very precise asymptotics of both the left tail and the right
tail. Their result for the left tail is that, as \xtoo,
with $\gam=(2-\frac{1}{\ln2})\qw$,
\begin{align}\label{knsz-}
  \P(Z\le -x) =(\cc+o(1))\ccdef\cco\exp\bigpar{-\cc e^{\gam x}}
\ccdef\ccz
=\exp\bigpar{-e^{\gam x +\cc+o(1)}},
\end{align}
where $\cco,\ccz,\ccx$ are some constants ($\cco$ is explicit in
\cite{KnSz}, but not $\ccz$).
For the right tail, they give a more complicated expression, which by
ignoring higher order terms implies, for example,
\begin{align}\label{knsz+}
  \P(Z\ge x)
=\exp\bigpar{-x\ln x-x\ln\ln x+(1+\ln 2)x+o(x)}.
\end{align}

It has been a challenge to justify these asymptotics rigorously, and
so far very little progress has been made. 
Some rigorous upper bounds were given by
\citet{SJ138}, in particular
\begin{align}\label{fj+}
  \P(Z\ge x) \le \exp\bigpar{-x\ln x +(1+\ln2)x},
\qquad x\ge 303,
\end{align}
with the same leading term (in the exponent) as \eqref{knsz+},
and for the left tail
\begin{align}
  \label{fj-}
\P(Z\le -x) \le \exp(-x^2/5),
\qquad x\ge0,
\end{align}
which is much weaker than \eqref{knsz-}.

Also the present paper falls  short of the (non-rigorous) asymptotics
\eqref{knsz-}--\eqref{knsz+} from \cite{KnSz},
but we show, by simple methods, the following results, which at least show
that the leading terms in the top exponents in \eqref{knsz-}--\eqref{knsz+} are
correct. 

\begin{theorem}
  \begin{thmenumerate}
  \item 
Let 
$\gam:=(2-\lniiw)\qw$.
  As $\xtoo$,
  \begin{equation}\label{sj-}
\exp\bigpar{-e^{\gam x+\ln\ln x+O(1)}}
\le 
	\P(Z\le -x) 
\le \exp\bigpar{-e^{\gam x+O(1)}}
  \end{equation}
\item 	
  As $\xtoo$,
  \begin{equation}\label{sj+}
\exp\bigpar{-x\ln x-x\ln\ln x+O(x)}
\le 
	\P(Z\ge x) 
\le \exp\bigpar{-x\ln x+O(x)}.
  \end{equation}
  \end{thmenumerate}
\end{theorem}

We show the lower bounds in Sections \ref{Slowerleft} and \ref{Slowerright},
and the upper bounds in Sections \ref{Supperleft} and \ref{Supperright}.
The lower bounds are proved by direct arguments using the identity
\eqref{recZ}; the upper bounds are proved by the standard method of first
estimating the moment generating function.

\begin{remark}
  The right inequality in \eqref{sj+} follows from the more precise
  \eqref{fj+}, where an explicit value is given for the implicit constant;
we include this part of \eqref{sj+} for completeness. (The proof in
\refS{Supperright} actually yields a better constant than \eqref{fj+} for
large $x$, see \eqref{ql}.)
We expect that, similarly, the implicit constants in the other parts of
\eqref{sj-}--\eqref{sj+} could be replaced by explicit bounds, using more
careful versions of the arguments and estimates below. However, in order to
keep the proofs simple, we have not attempted this.
\end{remark}

\begin{remark}
  We consider only the limiting random variable $Z$, and not $Z_n$ or $X_n$
  for finite $n$. Of course, the results for $Z$ imply corresponding results
  for the tails $\P(Z_n\le -x)$  and $\P(Z_n\ge x)$ for $n$ sufficiently
  large (depending on $x$), but we do not attempt to give any explicit
  results for finite $n$.
For some bounds for finite $n$, see \cite{SJ141} and (for large deviations)
\cite{McDH}.
\end{remark}

\begin{remark}
  Although we do not work with $Z_n$ for finite $n$, the proofs below
of the lower  bounds  can be  interpreted for finite $n$, saying that
we can obtain $Z_n\le -x$ with roughly
 the given probability (for large $n$)
by considering the event that in the first $\Theta(x)$
generations, all splits are close to balanced (with proportions
 $\frac12\pm x\qqw$, say); similarly, to obtain  
$Z_n\ge x$ we let there be one branch of length $\Theta(x)$
where all splits are extremely
unbalanced (with at most a fraction $(x\ln x)\qw$ on the other side).
The fact that we require an exponential number of splits to be extreme for
the lower tail, but only a linear number for the right tail, can be seen as
an explanation of the difference between the two tails, with the left tail
doubly exponential and the right tail roughly exponential.
\end{remark}

\section{Preliminaries}
Note that $g$ in \eqref{g}
is a continuous convex function on $\oi$, with maximum
$g(0)=g(1)=1$ and minimum $g(1/2)=1-2\ln2=-(2\ln2-1)<0$.

Let $\psi(t):=\E e^{tZ}$ be the moment generating function 
of  $Z$.
As said above, 
R\"{o}sler~\cite{Roesler} showed that $\psi(t)$ is finite for every real $t$.
The distributional identity \eqref{recZ} yields, by conditioning on $U$,
the functional equation
\begin{equation}\label{psi}
  \psi(t):=\E e^{tZ}=\intoi \psi(ut)\psi((1-u)t)e^{tg(u)}\dd u.
\end{equation}

We may replace $Z$ by the right-hand side of \eqref{recZ}; hence we may
without loss of generality assume the equality (not just in distribution)
\begin{equation}\label{recZ=}
  Z= UZ'+(1-U)Z''+g(U).
\end{equation}

\section{Left tail, lower bound}\label{Slowerleft}

\begin{proof}[Proof of lower bound in \eqref{sj-}]
  Let $\eps>0$ be so small that $g(\frac12+\eps)<0$, and let
  $a:=-g(\frac12+\eps)>0$. 
For any $z$, on the event 
\xcpar{$Z'\le-z$, $Z''\le -z$, and  $|U-\frac12|\le\eps$}, 
\eqref{recZ=} yields 
\begin{equation}
  Z\le -Uz-(1-U)z+g(U)=-z+g(U)\le -z-a  .
\end{equation}
Hence, for any real $z$,
\begin{equation}
  \P(Z\le -z-a) \ge 2\eps\P(Z\le-z)^2.
\end{equation}
It follows by induction that
\begin{equation}
  \P(Z\le -na) \ge (2\eps)^{2^n-1}\P(Z\le0)^{2^n},
\qquad n\ge0.
\end{equation}
Consequently, using $2\eps\le1$,
$  \P(Z\le -na) \ge (2\eps\P(Z\le0))^{2^n}$, and thus, with
$c:=\ln(2\P(Z\le0))>-\infty$,
\begin{equation}
  \ln\P(Z\le -na)\ge 2^n\bigpar{\ln\eps+c},
\qquad n\ge0.
\end{equation}
If $x>0$, we take $n=\ceil{x/a}$ and obtain
\begin{equation}\label{lukas}
  \ln\P(Z\le -x)\ge 2^{x/a+1}\bigpar{\ln\eps+c}.
\end{equation}
We choose (for large $x$) $\eps=x\qqw$, so, using Taylor's formula,
\begin{equation}
a=-g\bigpar{\tfrac12+\eps}=-g\bigpar{\tfrac12} +O\bigpar{\eps^2}
=2\ln2-1+O\bigpar{x\qw} 
\end{equation}
and thus
\begin{equation}
  a\qw 
=(2\ln2-1)\qw+O\bigpar{x\qw}.
\end{equation}
Consequently, \eqref{lukas} yields
\begin{equation}
  \ln\P(Z\le -x)\ge 2^{x/(2\ln2-1)+O(1)}\bigpar{\ln x\qqw+c}
=-e^{\gam x+O(1)+\ln\ln x}.
\end{equation}
\end{proof}

\section{Right tail, lower bound}\label{Slowerright}

\begin{proof}[Proof of lower bound in \eqref{sj+}]
  Let $0<\gd<\frac12$.
If $0<U\le \gd$, then 
\begin{equation}\label{pyret}
g(U)\ge g(\gd)=1+2\gd\ln\gd+O(\gd)
\ge 1+3\gd\ln\gd , 
\end{equation}
with the last inequality holding
provided $\gd$ is small enough.

Assume that \eqref{pyret} holds, and assume that $Z'\ge0$, $Z''\ge z\ge0$ and 
$U\le\gd$. Then \eqref{recZ=} yields
\begin{equation}
  Z\ge (1-\gd)z+g(\gd)\ge z-\gd z+1-3\gd\ln\gd\qw.
\end{equation}
Consequently,
\begin{equation}
  \P(  Z\ge z+1-\gd z-3\gd\ln\gd\qw)
\ge
\gd\P(Z\ge0)\P(Z\ge z).
\end{equation}

Let $x$ be sufficiently large and choose $\gd=1/(x\ln x)$.
Then, for $0\le z\le x$, 
\begin{equation}
  z+1-\gd z-3\gd\ln\gd\qw
\ge z+1-\frac{1}{\ln x}-3\frac{\ln (x\ln x)}{x\ln x}
\ge z+1-\frac{2}{\ln x},
\end{equation}
provided $x$ is large enough.
Hence, if $b:=1-\frac{2}{\ln x}$ and $c:=\P(Z\ge0)>0$, then for $0\le z\le x$
we have
\begin{equation}
  \P(Z\ge z+b ) \ge c\gd\P(Z\ge z).
\end{equation}
By induction, we find for $0\le n\le x/b+1$,
\begin{equation}
  \P(Z\ge nb)\ge c^n\gd^n \P(Z\ge0)=c^{n+1}\gd^n
> (c\gd)^{n+1}.
\end{equation}
Consequently, taking $n:=\ceil{x/b}$,
\begin{equation}
  \begin{split}
\ln\P(Z\ge x) &\ge 
(n+1)(\ln c+ \ln \gd)	
\ge (x/b+2)(\ln c+\ln \gd)
\\&
=\bigpar{x+O(x/\ln x)}\bigpar{-\ln x-\ln\ln x+O(1)}
\\&
=-x\ln x-x\ln\ln x+O(x).
  \end{split}
\end{equation}
\end{proof}

\section{Left tail, upper bound}\label{Supperleft}

\begin{lemma}\label{L-}
  There exists $a\ge0$ such that for all $t>0$,
with $\kk:=\gam\qw=2-\lniiw$, 
  \begin{equation}\label{6a}
	\psi(-t)< \exp\bigpar{\kk t\ln t+at+1}.
  \end{equation}
\end{lemma}

\begin{proof}
We 
note that $t\ln t\ge -e\qw$ for $t>0$, and thus
$\kk t\ln t+at+1\ge -\kk e\qw+1> 0$.
Since $\psi(t)$ is continuous and $\psi(0)=1$, there exists $t_1>0$ such
that $\psi(-t)< \exp\bigpar{1-\kk e\qw}$ for $0\le t\le t_1$, and thus
\eqref{6a} holds for all such $t$, and any $a\ge0$.
Next, let $t_2:=\pi e^{2}$. We may choose $a>0$ such that \eqref{6a} holds
for $t\in[t_1,t_2]$.

Before proceeding to larger $t$, define
\begin{equation}
  \label{h}
h(u):=u\ln u+(1-u)\ln(1-u)
\end{equation}
and note that $g(u)=2h(u)+1$ by \eqref{g}.

Now suppose that \eqref{6a} fails for some $t>0$ and let
$T:=\inf\set{t>0:\text{\eqref{6a} fails}}$.
Then $T\ge t_2$, and, by continuity,
  \begin{equation}\label{7a}
	\psi(-T)= \exp\bigpar{\kk T\ln T+aT+1}.
  \end{equation}
Furthermore, if $0<u<1$, then \eqref{6a} holds for $t=uT$ and $t=(1-u)T$,
and thus, recalling \eqref{h},
{\multlinegap=0pt
\begin{multline*}
	\psi(-uT)\psi\bigpar{-(1-u)T}
\\
	\begin{aligned}
&<\exp\bigpar{\kk uT\ln(uT)+\kk(1-u)T\ln((1-u)T)+auT+a(1-u)T+2}
\\&
=\exp\bigpar{\kk T\ln T+\kk \bigpar{u\ln u+(1-u)\ln(1-u)}T+aT+2}.
\\&
=\exp\bigpar{\kk T\ln T+\kk h(u)T+aT+2}.
	\end{aligned}  
\end{multline*}}
Furthermore, $g(u)=1+2h(u)$, and thus we obtain
\begin{equation}\label{winston}
  	\psi(-uT)\psi\bigpar{-(1-u)T}e^{-Tg(u)}
\\
\le\exp\bigpar{\kk T\ln T-((2-\kk) h(u)+1)T+aT+2}.
\end{equation}

By \eqref{h}, $h(u)$ is a convex function  with $h(\frac12)=-\ln2$,
$h'(\frac12)=0$ and $h''(u)=u\qw+(1-u)\qw\ge4$, and thus by Taylor's
formula,
$h(u)\ge -\ln2+2(u-\frac12)^2$.
Furthermore, $2-\kk=1/\ln2$, and thus
\begin{equation}\label{jw}
(2-\kk)h(u)+1\ge \frac{2}{\ln2}(u-\tfrac12)^2 
\ge (u-\tfrac12)^2 .
\end{equation}

Combining \eqref{psi}, \eqref{winston}, and \eqref{jw}, we obtain
\begin{equation}
  \begin{split}
  \psi(-T)
&\le\intoi \exp\Bigpar{\kk T\ln T + aT + 2 -(u-\tfrac12)^2T}\dd u
\\&
< \exp\bigpar{\kk T\ln T + aT + 2}\intoooo e^{ -(u-\frac12)^2T}\dd u	
\\&
= \sqrt{\frac{\pi}{T}} \exp\bigpar{\kk T\ln T + aT + 2}.
  \end{split}
\end{equation}
Since $T\ge t_2=\pi e^{2}$, this yields 
$\psi(-T)< \exp\bigpar{\kk T\ln T + aT + 1}$, which contradicts \eqref{7a}.
This contradiction shows that no such $T$ exists, and thus \eqref{6a} holds
for all $t>0$.
\end{proof}

\begin{proof}[Proof of upper bound in \eqref{sj-}]
For $x\ge0$ and any $t\ge0$, by \refL{L-},
\begin{equation}
  \P(Z\le -x)\le e^{-tx}\E e^{-tZ}=e^{-tx}\psi(-t)
<\exp\bigpar{-tx+\kk t\ln t+at +1}.
\end{equation}
We optimize by taking $t=\exp(\kk\qw(x-a)-1)$ and obtain
\begin{equation}
\ln  \P(Z\le -x)
<t(\kk \ln t+a-x) +1
=-\kk t+1=-e^{\kk\qw x+O(1)},
\end{equation}
which is the upper bound in \eqref{sj-}  because $\kk\qw=\gam$.
\end{proof}

\section{Right tail, upper bound}\label{Supperright}

As said in the introduction, \eqref{fj+} was proved in \cite{SJ138}.
Nevetheless we give for completeness a proof of the upper bound in
\eqref{sj+}, similar to the proof in \refS{Supperleft}. 
(It is also similar to the proof in \cite{SJ138} but simpler, partly
because we do not keep track of all constants and do not try to optimize;
nevertheless, it yields a slight improvement of \eqref{fj+} for large $x$,
see \eqref{ql} below.)

\begin{lemma}
  \label{L+}
There exists $a\ge0$ such that for all $t\ge0$,
\begin{equation}\label{2a}
  \psi(t)\le\exp\bigpar{e^t+at}.
\end{equation}
\end{lemma}

Note that \cite[Corollary 4.3]{SJ138} shows the bound $\psi(t)\le\exp(2e^t)$
for $t\ge 
5.02$, which is explicit, but  weaker for large $t$.

\begin{proof}
  Since $\psi(0)=1<e$, it follows by continuity that there exists $t_1>0$
  such that $\psi(t)\le e$ for $t\in[0,t_1]$, and thus \eqref{2a} holds for
  $t\in[0,t_1]$ and any $a\ge0$.

Let $t_2:=100$, and choose $a$ so that \eqref{2a} holds for $t\in[t_1,t_2]$.
Assume that \eqref{2a} fails for some
$t>0$, and 
let
$T:=\inf\set{t>0:\text{\eqref{2a} fails}}$.
Then $T\ge t_2$, and, by continuity,
  \begin{equation}\label{2c}
	\psi(T)= \exp\bigpar{e^T+aT}.
  \end{equation}
Furthermore, if $0<u<1$, then \eqref{2a} holds for $t=uT$ and $t=(1-u)T$,
and thus, using \eqref{psi} and the symmetry $u\leftrightarrow1-u$ there,
and $g(u)\le 1$,
\begin{equation}\label{2b}
  \begin{split}
\psi(T)&
\le2\int_0^{1/2} \exp\Bigpar{e^{uT}+auT+e^{(1-u)T}+a(1-u)T+Tg(u)}\dd u
\\&
\le2\int_0^{1/2} \exp\Bigpar{e^{uT}+e^{T-uT}+aT+T}\dd u.	
  \end{split}
\end{equation}
We consider two cases.

(i) If $uT\le 1$, then $e^{-uT}\le 1-\frac12uT$, and thus
\begin{equation}
  e^{uT}+e^{T-uT}+aT+T
\le e+e^T(1-\tfrac12uT)+(a+1)T.
\end{equation}
Hence, the contribution to \eqref{2b} for $u\le 1/T$ is no more than
\begin{equation}
  \begin{split}
	\label{3a}
2\int_0^{1/T} \exp\Bigpar{e^{T}+(a+1)T&+e-\tfrac12 Te^T u}\dd u
\\
&<2 \exp\Bigpar{e^{T}+(a+1)T+e}\frac{1}{\frac12 Te^T}
\\&
=\frac{4e^e}{T} \exp\bigpar{e^T+aT}
\le 0.7 \psi(T),
  \end{split}
\end{equation}
by \eqref{2c} and $T\ge t_2=100$, since $4e^e\doteq 60.62$.

(ii)
For $uT>1$ and $u<\frac12$, recalling $T\ge t_2=100$,
\begin{equation}
  \begin{split}
e^{uT}+e^{T-uT}+aT+T
&
\le 2e^{T-uT}+aT+T
\le 2e\qw e^{T}+aT+T
\\&
\le 0.8 e^{T}+T +aT
\le 0.9 e^T+aT
\\&
= e^T+aT-0.1 e^{T}
\le e^T+aT-100.
  \end{split}
\end{equation}
Hence, the contribution to \eqref{2b} for $uT>1$ is less than, 
recalling \eqref{2c},
\begin{equation}\label{4a}
  \exp\bigpar{e^T+aT-100}
=e^{-100}\psi(T) <0.1 \psi(T).
\end{equation}

Using \eqref{3a} and \eqref{4a} in \eqref{2b}, we find
\begin{equation}
  \psi(T)<0.7\psi(T)+0.1\psi(T),
\end{equation}
a contradiction. Hence $T$ cannot exist and \eqref{2a} holds for all $t\ge0$.
\end{proof}

\begin{proof}[Proof of upper bound in \eqref{sj+}]
For $x\ge0$ and any $t\ge0$, by \refL{L+},
\begin{equation}
  \P(Z\ge x)
\le e^{-tx}\E e^{tZ}
=e^{-tx}\psi(t)
\le\exp\bigpar{-tx+e^t+at}.
\end{equation}
We take $t=\ln x$ (assuming $x\ge1$)
and obtain
\begin{equation}\label{ql}
  \P(Z\ge x)
\le \exp\bigpar{-x\ln x+x+O(\ln x)},
\qquad x\ge1.
\end{equation}
(The optimal choice of $t$ is actually $\ln(x-a)$, but this leads to the
same result up to $o(1)$ in the exponent, 
which is absorbed by the error term $O(\ln x)$.)
\end{proof}

\begin{ack}
  I thank David Belius and Jim Fill for helpful comments.
\end{ack}

\newcommand\AAP{\emph{Adv. Appl. Probab.} }
\newcommand\JAP{\emph{J. Appl. Probab.} }
\newcommand\JAMS{\emph{J. \AMS} }
\newcommand\MAMS{\emph{Memoirs \AMS} }
\newcommand\PAMS{\emph{Proc. \AMS} }
\newcommand\TAMS{\emph{Trans. \AMS} }
\newcommand\AnnMS{\emph{Ann. Math. Statist.} }
\newcommand\AnnPr{\emph{Ann. Probab.} }
\newcommand\CPC{\emph{Combin. Probab. Comput.} }
\newcommand\JMAA{\emph{J. Math. Anal. Appl.} }
\newcommand\RSA{\emph{Random Struct. Alg.} }
\newcommand\ZW{\emph{Z. Wahrsch. Verw. Gebiete} }
\newcommand\DMTCS{\jour{Discr. Math. Theor. Comput. Sci.} }

\newcommand\AMS{Amer. Math. Soc.}
\newcommand\Springer{Springer-Verlag}
\newcommand\Wiley{Wiley}

\newcommand\vol{\textbf}
\newcommand\jour{\emph}
\newcommand\book{\emph}
\newcommand\inbook{\emph}
\def\no#1#2,{\unskip#2, no. #1,} 
\newcommand\toappear{\unskip, to appear}

\newcommand\arxiv[1]{\texttt{arXiv:#1}}
\newcommand\arXiv{\arxiv}

\end{document}